\documentclass[12pt]{amsart}

\usepackage{amssymb,amsmath,amsbsy, amsthm}
\usepackage[english]{babel}
\usepackage{amscd}
\usepackage{pb-diagram}
\usepackage{pstricks,pst-node}
\usepackage{amsfonts}
\usepackage{latexsym}
\usepackage{pstricks,pst-node}
\usepackage{tikz}
\usepackage{fullpage}

\usepackage{hyperref}

\newtheorem{theorem}{Theorem}[section]
\newtheorem{lemma}[theorem]{Lemma}
\newtheorem{corollary}[theorem]{Corollary}
\newtheorem{claim}[theorem]{Claim}

\newtheorem{proposition}[theorem]{Proposition}
\newtheorem{example}[theorem]{Example}
\newtheorem{question}[theorem]{Question}
\theoremstyle{definition}

\newtheorem{remark}[theorem]{Remark}

\numberwithin{equation}{section}


\newcommand{\N}{\mathbb{N}}

\begin{document}
\date{}
\title{Equicontinuity of maps on dendrites}         
\author{Javier Camargo   \and Michael Rinc\'on \and Carlos Uzc\'ategui}
\address{Escuela de Matem\'aticas, Facultad de Ciencias, Universidad Industrial de
	Santander, Ciudad Universitaria, Carrera 27 Calle 9, Bucaramanga,
	Santander, A.A. 678, COLOMBIA.}
\email{jcamargo@saber.uis.edu.co}

\address{Escuela de Matem\'aticas, Facultad de Ciencias, Universidad Industrial de
	Santander, Ciudad Universitaria, Carrera 27 Calle 9, Bucaramanga,
	Santander, A.A. 678, COLOMBIA.}
\email{marinvil@uis.edu.co}

\address{Escuela de Matem\'aticas, Facultad de Ciencias, Universidad Industrial de
	Santander, Ciudad Universitaria, Carrera 27 Calle 9, Bucaramanga,
	Santander, A.A. 678, COLOMBIA. Centro Interdisciplinario de L\'ogica y \'Algebra, Facultad de Ciencias, Universidad de Los Andes, M\'erida, VENEZUELA.}
\email{cuzcatea@saber.uis.edu.co}



\begin{abstract}
Given a dendrite $X$ and a continuous map $f\colon X\to X$, we show  the following are equivalent: (i) $\omega_f$ is continuous and $\overline{\mathrm{Per}(f)}=\bigcap_{n\in\N}f^n(X)$; (ii) $\omega(x,f)=\Omega(x,f)$ for each $x\in X$;  and (iii) $f$ is equicontinuous. Furthermore, we present some examples illustrating our results.
\end{abstract}

\maketitle

\section{Introduction}

There has been a lot of interest in the study of dynamical system defined on dendrites \cite{Mai,Sun,Sun2,Sun3}.  A particular interesting issue is to determine when a continuous map $f\colon X\to X$ is equicontinuous, i.e. when the collection $\{f^n:\,n\in\N\}$ of iterates of $f$ is an equicontinuous family.  The paradigmatic examples are given by continuous maps $f\colon [0,1]\to [0,1]$. In this case, it is known that  $f$ is equicontinuous if, and only if, $\bigcap_{n\in\mathbb N} f^n[0,1]$ is equal to $\mathrm{Fix}(f^2)$, the collection of fixed points of $f^2$, and it is also equivalent to require that $\mathrm{Fix}(f^2)$ is connected (see \cite{Bruckner}).  In order to state other characterizations of equicontinuity, we need to introduce two crucial  concepts:
\begin{enumerate}
\item $\omega(x,f)$  is  the set of all points $y\in X$ such that there exists an increasing sequence $(n_i)_{i\in\N}$ in $\N$ satisfying $\lim_{i\to \infty}f^{n_i}(x)=y$.

\item $\Omega(x,f)$ is  the set of all points $y\in X$ such that there exist a sequence $(x_i)_{i\in\N}$ which converges to $x$ and an increasing sequence $(n_i)_{i\in\N}$ in $\N$ satisfying $\lim_{i\to \infty}f^{n_i}(x_i)=y$.
\end{enumerate}

Bruckner and Ceder \cite{Bruckner} incorporated into the study of  the dynamics of interval maps  the function $\omega_f\colon [0,1]\to2^{[0,1]}$ defined by $\omega_f(x)=\omega(x,f)$. They showed that $f$ is equicontinuous if, and only if, $\omega_f$ is continuous. These results have been partially extended to dendrites, that is, locally connected continua without simple closed curves \cite{Mai,Sun,Sun2,Sun3}. For instance,  Sun et al. \cite[Theorem 2.8]{Sun} showed that if $X$ is a dendrite such that it has less that $2^{\aleph_0}$ end points and  $f\colon X\to X$ is a continuous map, then $f$ is equicontinuous if, and only if, $\omega(x,f^n)=\Omega(x,f^n)$ for all $x\in X$ and $n\in \N$. Also, in \cite[Theorem 2.8]{Sun2}, it is shown that if $X$ is a dendrite with finite branch points and $f\colon X\to X$ is a continuous map, then equicontinuity is equivalent to $\omega(x,f)=\Omega(x,f)$ for each $x\in X$. The main result in the present paper is Theorem \ref{MainTheorem}, where we show the following: Let $X$ be a dendrite, $f\colon X\to X$ be a continuous map and $\mathrm{Per}(f)$ be the collection of periodic points. The following are equivalent:
	\begin{enumerate}
		\item $\omega_f$ is continuous and $\overline{\mathrm{Per}(f)}=\bigcap_{n\in\N}f^n(X)$;
		\item $\omega(x,f)=\Omega(x,f)$ for each $x\in X$;
		\item $f$ is equicontinuous.
	\end{enumerate}

We present some examples showing  the necessity of some conditions as in (1).  We also present an example showing this result is no valid for fans.

\section{Definitions and preliminaries}

Let $Z$ be a metric space, then given $A\subseteq Z$ and $\epsilon>0$, the open ball about $A$ of radius $\epsilon$ is denoted by $\mathcal{V}(A,\epsilon)$. The  interior, clousure, boundary and cardinality of $A$ are denoted by $A^{\circ}, \overline{A}, \mathrm{Bd}(A)$ and $|A|$, respectively. A \textit{map} is a continuous function.
Given a compact metric space $X$, we denote by $2^X$ the set of all nonempty closed subsets of $X$, topologized by the Hausdorff metric which is defined as follows: for $A,B\in 2^X$, we set
\begin{gather*}
    \mathcal{H}(A,B)=\inf\{\epsilon>0\,:\,\text{$A\subseteq\mathcal V(B,\epsilon)$ and $B\subseteq\mathcal V(A,\epsilon)$}\}.
\end{gather*}

A \text{continuum} is a nonempty, compact, connected metric space. We say that a continuum $X$ is a \textit{simple closed curve} if $X$ is homeomorphic to $S^1=\{z\in \mathbb{C} : |z|=1\}$. A continuum $X$ is a \textit{dendrite} provided that $X$ is locally connected and does not contain a simple closed curve. It is clear that dendrites are uniquely arcwise connected continua. So, if $X$ is a dendrite and $x,y\in X$, we denote by $[x,y]$ the unique arc joining $x$ and $y$; also $(x,y)$ will denote $[x,y]\setminus\{x,y\}$.

Let $X$ be a continuum. A point $p\in X$ is called an \textit{end point of} $X$ provided that whenever  $U$ is open and $p\in U$, there exists an open set $V\subseteq U$ such that $x\in V$ and $|\mathrm{Bd}(V)|=1$. We denote by $\mathrm{End}(X)$ the collection of end points of $X$. Furthermore, $p$ is a \textit{cut point of} $X$, if $X\setminus \{p\}$ is disconnected.

Let $(X,d)$ be a metric space. A map $f\colon X\to X$ is called \textit{equicontinuous} provided that for each $\epsilon >0$, there exists $\delta>0$ such that $d(f^n(x),f^{n}(y))<\epsilon$, for each $n\in\N$, whenever $d(x,y)<\delta$ with $x, y \in X$.

Given a map $f\colon X\to X$ where $X$ is a compact metric space and $x\in X$, the orbit under $f$ is the set $\mathcal O(x,f)=\{f^n(x)\,:\,n\in\mathbb N\}$. We say that $x$ is a \textit{periodic point} of $f$ provided $f^k(x)=x$ for some $k\in\N$. We denote the collection of periodic points of $f$ by $\mathrm{Per}(f)$. If $f(x)=x$, then we say that $x$ is a \textit{fixed point} of $f$; the set of all fixed point of $f$ is denoted by $\mathrm{Fix}(f)$.  Given a compact metric space $X$, a map $f\colon X\to X$ and $x\in X$. The following concepts will play an important role.

\begin{enumerate}
\item $\omega(x,f)$ is the set of all points $y\in X$ such that there exists an increasing sequence $(n_i)_{i\in\N}$ in $\N$ satisfying $\lim_{i\to \infty}f^{n_i}(x)=y$.
\item $\Omega(x,f)$ is the set of all points $y\in X$ such that there exist a sequence $(x_i)_{i\in\N}$ which converges to $x$, and an increasing sequence $(n_i)_{i\in\N}$ in $\N$ satisfying $\lim_{i\to \infty}f^{n_i}(x_i)=y$.
\end{enumerate}

 A point $x\in X$ is said to be \textit{recurrent} provided that $x\in\omega(x,f)$. The set of all recurrent points of $X$ will be denoted by $\mathrm{R}(f)$.

The following easy result is well known.

\begin{proposition}\label{prop01}
Let $X$ be a compact metric space, $x\in X$ and $f\colon X\to X$ be a map. If $\mathcal{I}$ is either $\omega(x,f)$ or $\Omega(x,f)$,  then $\mathcal{I}$ is a nonempty compact subset of $X$ and $f(\mathcal{I})=\mathcal{I}$.
\end{proposition}

Proposition \ref{prop01} says that the map $\omega_f\colon X\to 2^X$ defined by $\omega_f(x)=\omega(x,f)$, for $x\in X$, is well defined.  In \cite{Bruckner} was initiated the study of the continuity of  $\omega_f$ for functions defined on the unit interval.

We recall some results about dendrites that we shall use.

\begin{theorem}\cite[Theorem 10.7]{Nadler}\label{theo00}
A nondegenerate continuum $X$ is a dendrite if, and only if, each point of $X$ is either a cut point of $X$ or an end point of $X$.
\end{theorem}

\begin{theorem}\cite[Corollary 10.5]{Nadler}\label{prop00}
Every subcontinuum of a dendrite is a dendrite.
\end{theorem}

A continuum $X$ is said to be \textit{regular} provided that each point $p\in X$ has a local base $\mathcal{B}_p$ such that $|\mathrm{Bd}(B)|<\infty$, for each $B\in\mathcal{B}_p$.

\begin{theorem}\cite[Theorem 10.20]{Nadler}\label{theo01}
Every dendrite is regular.
\end{theorem}

The following lemma is probably  known but we include a proof for sake of completeness.

\begin{lemma}\label{lemaf45f}
Let $X$ be a dendrite and let $p\in \mathrm{End}(X)$. If $A$ is a compact subset of $X$ such that $p\notin A$, then there exists a subcontinuum $W$ of $X$ such that $A\subseteq W\subseteq X\setminus\{p\}$.
\end{lemma}

\begin{proof}
Since $X$ is regular (see Theorem \ref{theo01}), there exists an open subset $U$ of $X$ such that $p\in U\subseteq X\setminus A$ and $|\mathrm{Bd}(U)|<\infty$. Let $T$ be the union of all posible arcs with end points in $\mathrm{Bd}(U)$. Note that $T$ is a continuum  such that $p\notin T$. Furthermore, each component of $X\setminus U$ intersects $T$, by \cite[Theorem 5.6]{Nadler}. Therefore, $W=T\cup (X\setminus U)$ is a continuum such that $A\subseteq W\subseteq X\setminus\{p\}$.
\end{proof}

\begin{theorem}\cite[Theorem 10.31]{Nadler}\label{fixpoint}
If $X$ is a dendrite and $f\colon X\to X$ is a map, then $\mathrm{Fix}(f)\neq\emptyset$.
\end{theorem}

The following theorem  provides a crucial property  of $\omega(x,f)$ when $X$ is a dendrite and $f$ a homeomorphism.

\begin{theorem}\cite[Theorem 3.8]{Acosta}
\label{acosta}
Let $X$ be a dendrite and $f\colon X\to X$ be a homeomorphism. For all $x\in X$, $\omega(x,f)$ is either a periodic orbit or a Cantor set.
\end{theorem}

\section{Some  results for compact metric spaces}

The main result of this paper is about dendrites but some of the auxiliary lemmas needed are valid in general for compact metric spaces, so we collect them in this section.

If $X$ is a compact metric space and $f\colon X\to X$ is a map, then $\{f^n(X) : n\in\N\}$ is a decreasing sequence of compacta; thus $\bigcap_{n\in\N}f^n(X)$ is a nonempty compact set. The next lemma is probably known but we include a simple proof for the convenience of the reader.

\begin{lemma}\label{claim00}
Let $X$ be a compact metric space and $f\colon X\to X$ be a map. If $J=\bigcap_{n\in\N}f^n(X)$, then $f(J)=J$.
\end{lemma}

\begin{proof}
	Observe that $f(J)\subseteq \bigcap_{n\in\N}f^{n+1}(X)=J$. Thus, $f(J)\subseteq J$. Conversely, let $y\in J$ be given. For each $n\in\N\setminus\{1\}$, let $x_n\in f^{n-1}(X)$ be such that $f(x_n)=y$. Since $X$ is compact, there exists a subsequence  $(x_{n_k})_{k\in\N}$ of $(x_n)_{n\in\N}$ such that $\lim_{k\to\infty}x_{n_k}=x_0$, for some $x_0\in X$.
	
	We show that $x_0\in J$. If $l\in\N$, then it is clear that $x_{n_k}\in f^{n_k-1}(X)\subseteq f^l(X)$, for each $n_k-1\geq l$. Since $f^l(X)$ is compact, $x_0\in f^l(X)$. Thus, $x_0\in J$. Finally, observe that $f(x_0)=y$ because $f(x_n)=y$, for each $n\in \N$. Therefore, $y\in f(J)$ and  we are done.
\end{proof}

\begin{lemma}\label{Lema032a}
Let $X$ be a compact metric space and let $f\colon X\to X$ be a map. If $\omega_f$ is continuous and $\mathrm{R}(f) \subseteq \overline{\mathrm{Per}(f)}$, then $\mathrm{R}(f)=\overline{\mathrm{Per}(f)}$.
\end{lemma}

\begin{proof}
Let $x\in \overline{\mathrm{Per}(f)}$ and  a sequence $(x_i)_{i\in\N}$ in $\mathrm{Per}(f)$ such that $\lim_{i\to\infty}x_i=x$. Since $\omega_f$ is continuous, $\lim_{i\to\infty}\omega(x_i,f)=\omega(x,f)$. Note that $x_i\in\omega(x_i,f)$, for each $i\in\N$. Therefore, $x\in\omega(x,f)$ and $x\in \mathrm{R}(f)$.
\end{proof}

The previous result can also be stated as follows.

\begin{lemma}\label{Lema032}
Let $X$ be a compact metric space and  $f\colon X\to X$ be a map. If $\omega_f$ is continuous and $\overline{\mathrm{Per}(f)}=\bigcap_{n\in\N}f^n(X)$, then $\mathrm{R}(f)=\bigcap_{n\in\N}f^n(X)$.
\end{lemma}

\begin{lemma}\label{Prop3d3}
Let $X$ be a compact metric space and  $f\colon X\to X$ be a map. If $\omega_f$ is continuous and $\overline{\mathrm{Per}(f)}=\bigcap_{n\in\N}f^n(X)$, then $\omega(x,f)=\Omega(x,f)$ for each $x\in X$.
\end{lemma}

\begin{proof}
Let $x\in X$ be given. Clearly $\omega(x,f)\subseteq \Omega(x,f)$. To  see that $\Omega(x,f)\subseteq \omega(x,f)$, let $y\in\Omega(x,f)$. Then there exist $(x_k)_{k\in\N}$ and $(n_k)_{k\in\N}$ in $\N$, such that $\lim_{k\to\infty}x_k=x$ and $\lim_{k\to\infty}f^{n_k}(x_k)=y$. Since $\omega_f$ is continuous, we have both $\lim_{k\to\infty}\omega(x_k,f)=\omega(x,f)$ and $\lim_{k\to\infty}\omega(f^{n_k}(x_k),f)=\omega(y,f)$. Notice that $\omega(x_k,f)=\omega(f^{n_k}(x_k),f)$ for each $k\in\N$. Hence, $\omega(x,f)=\omega(y,f)$. Since $\Omega(x,f)\subseteq \bigcap_{n\in\N}f^n(X)$, then $y\in R(f)$, by Lemma \ref{Lema032}. Therefore, $y\in\omega(y,f)=\omega(x,f)$ and  thus $\omega(x,f)=\Omega(x,f)$.
\end{proof}

The following was shown  in \cite{Sun} for dendrites.

\begin{lemma}
\label{equi-omega-iguales}
Let $(X,d)$ be a compact metric space and $f\colon X\to X$ equicontinuous. Then $\Omega(x,f^n)=\omega(x,f^n)$ for all $x\in X$ and $n\in \mathbb N$.
\end{lemma}

\begin{proof} It is not difficult to prove that if $f$ is equicontinuous, then so is $f^n$ for each $n\in \N$. Thus, it suffices to show that $\Omega(x,f)\subseteq\omega(x,f)$. Let $y\in\Omega(x,f)$ and $\varepsilon>0$ be given. Since $f$ is equicontinuous, there is
  $\delta>0$ such that if $x,y\in X$ and $d(x,y)<\delta$, then $d(f^n(x),f^n(y))<\varepsilon$ for all $n\in\mathbb N$. From definition of $\Omega(x,f)$, there are sequences $(n_k)_{k\in\mathbb N}$ in $\mathbb N$ and $(x_k)_{k\in\mathbb N}$ in $X$ satisfying $\lim_{k\to\infty}x_k= x$ and $\lim_{k\to\infty}f^{n_k}(x_k)= y$. So, there exists $k_0\in\mathbb N$ such that
  $d(x_k,x)<\delta$ for all $k\geq k_0$. Thus, $d(f^{n_k}(x_k),f^{n_k}(x))<\varepsilon$ for all $k\geq k_0$. Whence, $\lim_{k\to\infty}f^{n_k}(x)= y$, that is, $y\in\omega(x,f)$.
\end{proof}

The proof of the following result is the same as \cite[Theorem 1.2]{Bruckner} where it was shown for the interval.

\begin{theorem}\label{theo946t}
Let $X$ be a compact metric space and  $f\colon X\to X$ be a map. If $f$ is equicontinuous, then $\omega_f$ is continuous.
\end{theorem}

\begin{proof}
Let $\epsilon >0$. Since $f$ is equicontinuous, there exists $\delta>0$ such that $d(f^n(x),f^n(y))<\epsilon/3$ for each $n\in\N$, whenever $d(x,y)<\delta$. Let $x,y\in X$ be such that $d(x,y)<\delta$ and $x_0\in \omega(x,f)$. There exists $(n_k)_{k\in\N}$ in $\N$ such that $d(f^{n_k}(x),x_0)<\frac{\epsilon}{3}$, for each $k\in\N$. Then $$d(f^{n_k}(y),x_0)\leq d(f^{n_k}(y),f^{n_k}(x))+d(f^{n_k}(x),x_0)<\frac{2}{3}\epsilon.$$ So, by passing to a convergent subsequence, we see that there is $y_0\in \omega(y,f)$ such that $d(x_0,y_0)<\epsilon$. Hence, $\omega(x,f)\subseteq \mathcal{V}(\omega(y,f),\epsilon)$. Similarly, we prove that $\omega(y,f)\subseteq \mathcal{V}(\omega(x,f),\epsilon)$. Thus, $\mathcal{H}(\omega(x,f),\omega(y,f))<\epsilon$ whenever $d(x,y)<\delta$. Therefore, $\omega_f$ is continuous.
\end{proof}

For dendrites we shall see (Theorem \ref{MainTheorem}) that  $f$ is equicontinuous if, and only if, $\omega(x,f)=\Omega(x,f)$ for each $x\in X$. So, it is natural to ask the following. 

\begin{question}
Let $X$ be a compact metric space and  $f\colon X\to X$ be a map. Suppose $\omega(x,f)=\Omega(x,f)$ for each $x\in X$, is $\omega_f$ continuous?
\end{question}

We have a partial answers.

\begin{proposition}
\label{Prop857u}
Let $X$ be a metric compact space, and let $f\colon X\to X$ be a map. If $\omega(x,f)=\Omega(x,f)$ for each $x\in X$, then $\omega_f$ is usc.
\end{proposition}

\begin{proof}
Let $x_0\in X$ and $V$ be an open subset of $X$ such that $\omega_f(x_0)\subseteq V$. Suppose that there exists a sequence $(x_n)_n\subseteq X$ such that $\lim_{n\to\infty}x_n=x_0$ and $\omega_f(x_n)\cap (X\setminus V)\neq\emptyset$, for each $n\in\N$. Let $z_n\in \omega_f(x_n)\cap (X\setminus V)$. Since $X$ is compact, without loss of generality we may suppose that $\lim_{n\to\infty}z_n=z_0$ for some $z_0\in X\setminus V$. For each $n\in\N$, let $(k_i^n)_i\subseteq \N$ such that $\lim_{i\to\infty}f^{k_i^n}(x_n)=z_n$. Since $\lim_{n\to\infty}z_n=z_0$, it is not difficult to show that there exist $(n_k)_k, (l_k)_k\subseteq \N$  with $(l_k)_k$ increasing such that $\lim_{k\to\infty}f^{l_k}(x_{n_k})=z_0$. Since $\lim_{n\to\infty}x_n=x_0$, then $z_0\in\Omega(x_0,f)=\omega(x_0,f)$, thus $z_0\in V$, a contradiction. Therefore, $\omega_f$ is usc.
\end{proof}

\begin{proposition}
Let $X$ be a metric compact space, and let $f\colon X\to X$ be a map. If $X=\mathrm{Per}(f)$, then $\omega_f$ is lsc.
\end{proposition}

\begin{proof}
Let $(x_n)_n\subseteq X$ be such that $\lim_{n\to\infty}x_n=x_0$. Let $V$ be an open subset of $X$ such that $V\cap \omega_f(x_0)\neq\emptyset$. Let $z\in V\cap \omega_f(x_0)$. Since $x_0\in\mathrm{Per}(f)$, there is $m\in\N$ such that $f^m(x_0)=z$.  Since $\lim_{n\to\infty}f^m(x_n)=f^m(x_0)=z$, then there exists a positive integer $k$ such that $f^m(x_i)\in V$, for each $i\geq k$. Since $x_i\in \mathrm{Per}(f)$, $f^m(x_i)\in\omega_f(x_i)$ for each $i\in\N$. Thus, $\omega_f(x_i)\cap V\neq\emptyset$, for each $i\geq k$. Therefore, $\omega_f$ is lsc.
\end{proof}

From the results above we immediately get the following. 

\begin{theorem}
Let $X$ be a metric compact space and   $f\colon X\to X$ be a map. Suppose  $X=\mathrm{Per}(f)$. The following are equivalent:

\begin{itemize}
\item[(i)] 
 $\Omega(x,f)=\omega(x,f)$, for each $x\in X$.
 \item[(ii)]  $\omega_f$ is continuous.
 \item[(iii)]  $\omega_f$ is usc.
\end{itemize}
\end{theorem}

We will shall see in Theorem \ref{todos-peri} that for dendrites if every point is periodic, then $f$ is equicontinuous.

\section{Equicontinuity of maps on dendrites}

The main result on this section is Theorem \ref{MainTheorem} where we characterize the equicontinuity of maps defined in a dendrite. We need first to show several auxiliary results. The idea for the proof of the next proposition was taken from \cite[Lemma 2.3]{Sun}, they just proved that $\mathrm{Fix}(f)$ is connected.

\begin{lemma}\label{Prop2}
Let $X$ be a dendrite and $f\colon X\to X$ be a continuous map. If $\omega(x,f)=\Omega(x,f)$ for each $x\in X$, then $\mathrm{Fix}(f^m)$ is connected for all $m\in\mathbb N$. Moreover , $\mathrm{Per}(f)$ is also connected.
\end{lemma}

\begin{proof}
	Let $a,b\in \mathrm{Fix}(f^m)$, i.e. $f^m(a)=a$ and $f^m(b)=b$. We shall prove that $[a,b]\subseteq \mathrm{Fix}(f^m)$. Let $c\in [a,b]$ be given. Since $X$ is a dendrite, we have that either $c\in [a,f^m(c)]$ or $c\in [b,f^m(c)]$. Assume that $c\in [a,f^m(c)]$. Note that $[a,f^m(c)]\subseteq f^m([a,c])$. Thus, there exists $x_1\in [a,c]$ such that $f^m(x_1)=c$. Now, $[a,c]\subseteq f^m([a,x_1])$. Then there exists $x_2\in [a,x_1]$ such that $f^m(x_2)=x_1$. Proceeding inductively we construct a sequence $(x_n)_{n\in\N}$ such that $f^m(x_{i+1})=x_i$ and $[a,x_{i+1}]\subseteq [a,x_i]$, for each $i\in\N$. Therefore, $(x_n)_{n\in\N}$ is a convergent sequence; say $\lim_{n\to\infty}x_n=y_0$ for some $y_0\in [a,b]$. Clearly $f^{m}(y_0)=y_0$. Since $(f^m)^n(x_n)=c$ for each $n\in\N$, we have $c\in \Omega(y_0,f^m)\subseteq \Omega(y_0,f)=\omega(y_0,f)$. So, we infer from the fact $f^m(y_0)=y_0$ that  $f^m(c)=c$. Therefore, $[a,b]\subseteq \mathrm{Fix}(f^m)$ and $\mathrm{Fix}(f^m)$ is connected. Finally, connectedness of $\mathrm{Per}(f)$ follows immediately since $\mathrm{Per}(f)=\bigcup_{m\in\mathbb N}\mathrm{Fix}(f^m)$ and $\mathrm{Fix}(f)\subseteq \mathrm{Fix}(f^m)$ for all $m\in\mathbb N$.
	\end{proof}

The idea of the following result was taken from the proof of \cite[Lemma 2.7]{Sun2}.

\begin{lemma}\label{LemEq}
	Let $X$ be a dendrite and let $f\colon X\to X$ be a continuous map. If $\Omega(x,f)$ is totally disconnected for each $x\in X$, then $f$ is equicontinuous.
\end{lemma}

\begin{proof}
	Suppose that $f$ is not equicontinuous, i.e., there exist $x\in X$ and sequences $(x_k)_{k\in\N}$ in $X$ and $(n_k)_{n\in\N}$ in $\N$, such that $\lim_{k\to\infty}x_k=x$, $\lim_{k\to\infty}f^{n_k}(x_k)=a$ and $\lim_{k\to\infty}f^{n_k}(x)=b$, for some $a,b\in X$ with  $a\neq b$. Note that $a\in \Omega(x,f)$ and $b\in\omega(x,f)\subseteq \Omega(x,f)$. We show that $[a,b]\subseteq \Omega(x,f)$. Let $c\in (a,b)$ be given. Since $X$ is a dendrite, we have:
	\begin{enumerate}
	\item there exists $k_0\in\N$ such that $c\in [f^{n_k}(x_k),f^{n_k}(x)]$, for each $k\geq k_0$, and
	\item $\lim_{n\to\infty}t_n=x$, for every sequence $(t_n)_{n\in\N}$, where $t_n\in [x_n,x]$ for each $n\in\N$.
	\end{enumerate}
Observe that $[f^{n_k}(x_k),f^{n_k}(x)]\subseteq f^{n_k}([x_k,x])$ for each $k\in\mathbb N$. Hence, there exists $z_k\in [x_k,x]$ such that $f^{n_k}(z_k)=c$, for each $k\geq k_0$. In view of $\lim_{k\to\infty}z_k=x$ we deduce that $c\in\Omega(x,f)$. Thus, $[a,b]\subseteq \Omega(x,f)$. Therefore, $\Omega(x,f)$ is not totally disconnected.
\end{proof}

\begin{lemma}\label{Lemeq0}
	Let $X$ be a dendrite and let $f\colon X\to X$ be an onto map. If $\overline{\mathrm{Per}(f)}=X$ and $\mathrm{Per}(f)$ is connected, then $f$ is a homeomorphism.
\end{lemma}

\begin{proof}
         We show first that $X\setminus\mathrm{Per}(f)\subseteq \mathrm{End}(X)$. Note that if $x\in X\setminus (\mathrm{End}(X)\cap \mathrm{Per}(f))$, then $X\setminus\{x\}$ is disconnected, by Theorem \ref{theo00}. Hence, $X\setminus\{x\}=U\cup V$, where $U$ and $V$ are nonempty disjoint open subsets of $X$. Since $\mathrm{Per}(f)$ is connected and $\mathrm{Per}(f)\subseteq X\setminus\{x\}$, we have either $\mathrm{Per}(f)\subseteq U$ or $\mathrm{Per}(f)\subseteq V$. Furthermore, as $\mathrm{Per}(f)$ is dense, then  either $U=\emptyset$ or $V=\emptyset$; a contradiction. Therefore, $X\setminus\mathrm{Per}(f)\subseteq \mathrm{End}(X)$.
	
Now we prove that $f$ is injective. Suppose that there exist $x_0,y_0\in X$ such that $x_0\neq y_0$ and $f(x_0)=f(y_0)$. Observe that $f([x_0,y_0])$ is a subcontinuum of $X$. By \cite[Theorem 6.6]{Nadler} there are at least two non-cut points in $f([x_0,y_0])$. Thus, let $w_0\in f([x_0,y_0])$ be a non-cut point such that $w_0\neq f(x_0)$ and take $z_0\in (x_0,y_0)$ satisfying $f(z_0)=w_0$. Let $\alpha$ be an arc such that $\alpha\setminus \{z_0\}$ is disconnected and $\alpha\subseteq (x_0,y_0)$. Note that $\alpha\cap \mathrm{End}(X)=\emptyset$, thus  $\alpha\subseteq \mathrm{Per}(f)$. It is not difficult to see that $f|_{\mathrm{Per}(f)}$ is injective. Thus, $f|_{\alpha}\colon \alpha\to f(\alpha)$ is a homeomorphism. But this contradicts the fact that $\alpha\setminus\{z_0\}$ is disconnected and $f(\alpha)\setminus\{f(z_0)\}$ is connected. Therefore, $f$ is injective and so, $f$ a homeomorphism.
\end{proof}

\begin{question}
Let $X$ be a dendrite and $f\colon X\to X$ be an onto map such that  $\mathrm{Per}(f)$ is connected and dense in $X$. Does it follow that $\omega_f$ is continuous?
\end{question}

We need the following result which is a consequence of \cite[Lemma 2.6]{Sun}.

\begin{lemma}
\label{minimal}
Let $X$ be a dendrite and let $f\colon X\to X$ be a map. If $\omega(x,f)=\Omega(x,f)$ for each $x\in X$, then for all $x\in X$, we have $\omega(y,f)=\omega(x,f)$ for all $y\in \omega(x,f)$.
\end{lemma}

\begin{theorem}\label{Teo908i71}
Let $X$ be a dendrite and let $f\colon X\to X$ be an onto map. If $\Omega(x,f)=\omega(x,f)$ for each $x\in X$, then $\overline{\mathrm{Per}(f)}=X$.
\end{theorem}

\begin{proof}
Suppose that $X\setminus \overline{\mathrm{Per}(f)}\neq\emptyset$. Since $f(\mathrm{Per}(f))=\mathrm{Per}(f)$, it is easy to see that \begin{equation}\label{eq000} f(\overline{\mathrm{Per}(f)})=\overline{\mathrm{Per}(f)}. \end{equation}

\begin{claim}\label{Claim5tg5}
$\omega(x,f)\cap \overline{\mathrm{Per}(f)}=\emptyset$, for each $x\in X\setminus \overline{\mathrm{Per}(f)}$.
\end{claim}

Suppose the contrary that there exists $x\in X\setminus \overline{\mathrm{Per}(f)}$ such that $\omega(x,f)\cap \overline{\mathrm{Per}(f)}\neq\emptyset$, i.e., there is $y_0\in\overline{\mathrm{Per}(f)}$ and an increasing sequence $(n_k)_{k\in\N}$ in $\N$ such that $\lim_{k\to\infty}f^{n_k}(x)=y_0$. Since $f$ is onto, it is not difficult to construct a sequence $(x_n)_{n\in\N}$ in $X$ such that  $f(x_1)=x$ and $f(x_{i+1})=x_i$, for each $i\in\N$. By compactness of $X$, there exists a subsequence $(x_{n_i})_{i\in\N}$ of $ (x_n)_{n\in\N}$ such that $\lim_{i\to\infty}x_{n_i}=z_0$, for some $z_0\in X$. Observe that $x\in \Omega(z_0,f)$ in view of $f^{n_i}(x_{n_i})=x$ for each $i\in\N$, and  $\lim_{i\to\infty}x_{n_i}=z_0$. So, we conclude that $x\in\omega(z_0,f)=\Omega(z_0,f)$. By applying Lemma \ref{minimal} twice, it follows that
\begin{gather*}
    \omega(x,f)=\omega(z_0,f),\quad\text{and}\quad\omega(y_0,f)=\omega(z_0,f).
\end{gather*}
Thus, $x\in \omega(y_0,f)$. But $\omega(y_0,f)\subseteq \overline{\mathrm{Per}(f)}$, by equation \eqref{eq000}. This contradicts the fact that $x\notin \overline{\mathrm{Per}(f)}$. Therefore, $\omega(x,f)\cap \overline{\mathrm{Per}(f)}=\emptyset$ for each $x\in X\setminus \overline{\mathrm{Per}(f)}$. We have completed the proof of Claim \ref{Claim5tg5}.

\begin{claim}\label{Claim6rh6}
There exists a component $W$ of $X\setminus \overline{\mathrm{Per}(f)}$ such that $f^m(W)\subseteq W$ for some $m\in\N$.
\end{claim}

Let $x\in X\setminus \overline{\mathrm{Per}(f)}$ and $z\in \omega(x,f)$ be fixed. Then there is a sequence $(n_k)_{k\in\N}$ in $\N$ such that $\lim_{k\to\infty}f^{n_k}(x)=z$. By Claim \ref{Claim5tg5}, $z\in X\setminus \overline{\mathrm{Per}(f)}$. Let $W$ be the component of $X\setminus \overline{\mathrm{Per}(f)}$ such that $z\in W$. Locally connectedness of $X$ implies that $W$ is open. Hence, there exists $k_0\in\N$ such that $f^{n_k}(x)\in W$ for each $k\geq k_0$. Let $s,t\in\mathbb N$ be such that $s,t\geq k_0$ and $n_s<n_t$. Notice that $f^{n_s}(x)\in W$ and $f^{n_t-n_s}(f^{n_s}(x))=f^{n_t}(x)\in W$. So, if $m=n_t-n_s$,  then $f^m(W)\cap W\neq\emptyset$. By Claim \ref{Claim5tg5} and equation \eqref{eq000} we have $f^m(W)\cap\overline{\mathrm{Per}(f)}=\emptyset$.  Thus, $f^m(W)\subseteq W$ as $f^m(W)$ is connected. We have completed the proof of  Claim \ref{Claim6rh6}.

\bigskip

         Since $W$ is a component of $X\setminus \overline{\mathrm{Per}(f)}$, $\overline{W}\cap\overline{\mathrm{Per}(f)}\neq\emptyset$, by \cite[Theorem 5.4]{Nadler}. Note that $\overline{W}$ is a dendrite (by Theorem \ref{prop00}).

 \begin{claim}\label{newclaim}
         There exists $p_0\in X$ such that $\overline{W}\cap\overline{\mathrm{Per}(f)}=\{p_0\}$, i.e., $\overline{W}=W\cup\{p_0\}$, where $p_0\in \mathrm{End}(\overline{W})$.
         \end{claim}
         Suppose that $\{x,y\}\subseteq \overline{W}\cap\overline{\mathrm{Per}(f)}$, where $x\neq y$. Since $W$ is connected, $W\cup \{x,y\}$ is connected. Hence, there exists an arc $[x,y]\subseteq W\cup\{x,y\}$, because every connected subset of a dendrite is arcwise connected (see \cite[Proposition 10.9]{Nadler}). Now, observe that $\overline{\mathrm{Per}(f)}$ is a dendrite by Lemma \ref{Prop2} and Theorem \ref{prop00}. So, $[x,y]\subseteq \overline{\mathrm{Per}(f)}$. This contradicts the fact that $W\cap \overline{\mathrm{Per}(f)}=\emptyset$. Then we have completed the proof of Claim \ref{newclaim}.


         Let $x\in W$ be given. By Claim \ref{Claim6rh6}, $f^{km}(x)\in W$ for each $k\in\N$. By Claim \ref{Claim5tg5}, $\omega(x,f^m)\subseteq W$. We set
         \begin{equation}
         H=\bigcap\{L\in\mathcal{C}(X) : \omega(x,f^m)\subseteq L\}.
         \end{equation}
         Since $X$ is dendrite, $H$ is a continuum (see \cite[Theorem 10.10]{Nadler}). Furthermore, $p_0\notin H$, by Lemma \ref{lemaf45f}. Thus, $H\subseteq W$. By Proposition \ref{prop01}, $f^{m}(\omega(x,f^m))=\omega(x,f^m)$ and so $H\subseteq f^m(H)$. Then,  $\{f^{km}(H)\,:\, k\in\N\}$ is an increasing sequence of continua.  By Claim \ref{Claim5tg5} we have $p_0\notin f^{km}(H)$ for each $k\in\N$. Henceforward
         by Claim \ref{Claim6rh6}, $f^{km}(H)\subseteq W$ for each $k\in\N$.

         Let $R=\overline{\bigcup_{k\in\N}f^{km}(H)}$. Observe that $R\subseteq W\cup\{p_0\}$ and $f^m(R)=R$. We shall show that $p_0\notin R$. If not, there exists a sequence $(y_n)_{n\in\N}$ in $ \bigcup_{k\in\N}f^{km}(H)$ such that $\lim_{n\to\infty}y_n=p_0$. For each $n\in\N$, let $x_n\in H$ be such that $f^{l_nm}(x_n)=y_n$ for some $l_n\in\N$. Since $H$ is compact, there exists a subsequence $(x_{n_k})_{k\in\N}$ of $ (x_n)_{n\in\N}$ such that $\lim_{k\to\infty} x_{n_k}=x_0$, for some $x_0\in H$. Thus, $p_0\in\Omega(x_0,f^m)\subseteq \Omega(x_0,f)=\omega(x_0,f)$, contradicting Claim \ref{Claim5tg5}. Therefore, $p_0\notin R$.

         Finally, observe that $R$ is a dendrite, by Theorem \ref{prop00}. Thus, $f^m|_R\colon R\to R$ is a map such that $\mathrm{Fix}(f^m|_R)=\emptyset$, since $\mathrm{Fix}(f^m|_R)\subseteq \mathrm{Per}(f)\subseteq X\setminus R$. This contradicts Theorem \ref{fixpoint}. Therefore, $\overline{\mathrm{Per}(f)}=X$.
\end{proof}

Next result shows that in the previous theorem the hypothesis that $f$ is onto is not necessary.

\begin{corollary}\label{Croequig}
Let $X$ be a dendrite and $f\colon X\to X$ be a map. If $\Omega(x,f)=\omega(x,f)$ for each $x\in X$, then $\overline{\mathrm{Per}(f)}=\bigcap_{n\in\N}f^n(X)$.
\end{corollary}

\begin{proof}
Let $J=\bigcap_{n\in\N}f^n(X)$ and $g=f|_J\colon J\to J$. By Lemma \ref{claim00} $g$ is an onto map. If $z\in J$, it is clear that:
\begin{itemize}
\item $\omega(z,f)=\omega(z,g)$;
\item $\omega(z,g)\subseteq \Omega(z,g)$; and
\item $\Omega(z,g)\subseteq \Omega(z,f)$.
\end{itemize}
Thus, $\omega(x,g)=\Omega(x,g)$, for each $x\in J$. Since $\mathrm{Per}(f)\subseteq J$, we have  $\overline{\mathrm{Per}(g)}=\overline{\mathrm{Per}(f)}$. Therefore, $\overline{\mathrm{Per}(f)}=\bigcap_{n\in\N}f^n(X)$, by Theorem \ref{Teo908i71}.
\end{proof}



The following result generalizes Lemma 2.4 of \cite{Sun}.


\begin{lemma}\label{TeoEqDen}
Let $X$ be a dendrite and $f\colon X\to X$ be a map. If $\Omega(x,f)=\omega(x,f)$ for each $x\in X$, then $f$ is equicontinuous.
\end{lemma}

\begin{proof}
Let $J=\bigcap_{n\in\N}f^n(X)$. By Corollary \ref{Croequig} and Lemma \ref{Prop2}, $\overline{\mathrm{Per}(f)}=J$ and $\mathrm{Per}(f)$ is connected, thus $J$ is a dendrite by Proposition \ref{prop00}. So, $f|_J\colon J\to J$ is a homeomorphism by Lemma \ref{Lemeq0}. By Theorem \ref{acosta},  $\omega(y,f)$ is totally disconnected for each $y\in J$. Let $x\in X$ and $y\in \omega (x,f)$. Notice that $y\in J$ as $\omega(x,f)\subseteq J$. By Lemma \ref{minimal} we have  $\omega (x,f)=\omega (y,f)$ and then $\omega (x,f)$ is totally disconnected. As $\omega(x,f)=\Omega(x,f)$ for each $x\in X$, then $f$ is equicontinuous, by Lemma \ref{LemEq}.
\end{proof}

The following theorem generalizes Theorem 2.8 of \cite{Sun} and Theorem 2.8 of \cite{Sun2}, and it comprises all results on this section.

\begin{theorem}\label{MainTheorem}
Let $X$ be a dendrite and let $f\colon X\to X$ be a map. The following are equivalent:
	\begin{enumerate}
		\item $\omega_f$ is continuous and $\overline{\mathrm{Per}(f)}=\bigcap_{n\in\N}f^n(X)$;
		\item $\omega(x,f)=\Omega(x,f)$ for every  $x\in X$;
		\item $f$ is equicontinuous.
	\end{enumerate}
\end{theorem}

\begin{proof}
By Lemmas \ref{equi-omega-iguales} and  \ref{TeoEqDen} we have the equivalence between (2) and (3).
 Since (2) and (3) are equivalent, by  Theorem \ref{theo946t} and Corollary \ref{Croequig} we have that (2) implies (1). Finally, Lemma \ref{Prop3d3} shows that (1) implies (2).
\end{proof}

From Lemma \ref{equi-omega-iguales} and Theorem \ref{MainTheorem} we immediately get the following.

\begin{corollary}
Let $X$ be a dendrite and let $f\colon X\to X$ be a map. If   $\omega(x,f)=\Omega(x,f)$ for all $x\in X$, then
$\omega(x,f^n)=\Omega(x,f^n)$ for all $x\in X$ and $n\in\N$.
\end{corollary}
We finish this section showing a necessary condition to have equicontinuity in dendrites.

\begin{theorem}\label{todos-peri}
Let $X$ be a dendrite and $f\colon X\to X$ be a map. If $\mathrm{Per}(f)=X$, then $f$ is equicontinuous.
\end{theorem}

\begin{proof}
By Theorem \ref{MainTheorem}, it suffices to show that $\omega(x,f)=\Omega(x,f)$ for each $x\in X$.
Let $x\in X$ be given and suppose that $\Omega(x,f)\setminus\omega(x,f)\neq\emptyset$. If $y\in \Omega(x,f)\setminus\omega(x,f)$, there exist a sequence $(x_n)_{n\in\N}$ in $ X$ and an increasing sequence $(k_n)_{n\in\N}$ in $\N$ such that $\lim_{n\to\infty}x_n=x$ and $\lim_{n\to\infty}f^{k_n}(x_n)=y$.

Since $x\in \mathrm{Per}(f)$, then $\omega(x,f)=\mathcal{O}(x,f)=\{x,f(x),...,f^{m-1}(x)\}$ for some positive integer $m$. Without loss of generality, we may assume that $k_n\equiv 1\ \mathrm{mod}(m)$, for each $n\in\N$. Let $c\in (f(x),y)$. As $X$ is a dendrite, there exists $n_0\in\N$ such that $c\in [f(x),f^{k_n}(x_n)]$, for each $n\geq n_0$. Note that $[f(x),f^{k_{n_0}}(x_{n_0})]\subseteq f^{k_{n_0}}([x,x_{n_0}])$. So, there is $z_0\in [x,x_{n_0}]$ such that $f^{k_{n_0}}(z_0)=c$. In view of $z_0\neq x$, there exists $n_1>n_0$ such that $z_0\notin [x,x_{n_1}]$. Now, as $[f(x),f^{k_{n_1}}(x_{n_1})]\subseteq f^{k_{n_1}}([x,x_{n_1}])$ and therefore, there exists $z_1\in [x,x_{n_1}]$ such that $f^{k_{n_1}}(z_1)=c$. By proceeding inductively we construct an infinite set $\{z_i : i\in\N\}$ such that $c\in \mathcal{O}(z_i,f)$ for each $i\in\N$. This contradicts the fact that $z_i\in\mathrm{Per}(f)$ for each $i\in\N$. Hence $\omega(x,f)=\Omega(x,f)$.
\end{proof}

%

\begin{remark}
Example \ref{EjemploCantor} shows an onto and equicontinuous map $f\colon X\to X$ such that $\mathrm{Per}(f)\neq X$.
\end{remark}

\section{Some examples}

 The following example shows that $\overline{\mathrm{Per}(f)}=\bigcap_{n\in\N}f^n(X)$ cannot be removed from the hypothesis of  (1) in Theorem \ref{MainTheorem}.

\begin{example}\label{Ex0}
	There exist a dendrite $X$ and a homeomorphism $f\colon X\to X$ such that $\omega_f$ is continuous and $\overline{\mathrm{Per}(f)}\neq X$. Furthermore, $\omega(x,f)\neq\Omega(x,f)$ for some $x\in X$, and $f$ is not equicontinuous.
\end{example}

Let $Z=J\cup (\bigcup_{n\in\mathbb{Z}}A_n)$, where:
\begin{itemize}
	\item $J=[-1,1]\times \{0\}$;
	\item $A_n=\{1-1/(n+1)\}\times [0,1/(n+1)]$, if $n\geq 0$;
	\item $A_n=\{-1+1/(1-n)\}\times [0,1/(1-n)]$, if $n<0$.
\end{itemize}
It is easy to see that $Z$ is a dendrite (see Figure 1).

\begin{center}
\begin{tikzpicture}[scale=3]
\draw (-1,0) to (1,0);
\draw (0,0) to (0,1);
\draw (0.5,0) to (0.5,0.5);
\draw (0.75,0) to (0.75,0.25);
\draw (0.833,0) to (0.833,0.1666);
\draw (0.89,0) to (0.89,0.11);
\fill (0.93,0.07) circle (0.2pt);
\fill (0.95,0.05) circle (0.2pt);
\fill (0.97,0.03) circle (0.2pt);

\draw (-0.5,0) to (-0.5,0.5);
\draw (-0.75,0) to (-0.75,0.25);
\draw (-0.833,0) to (-0.833,0.1666);
\draw (-0.89,0) to (-0.89,0.11);
\fill (-0.93,0.07) circle (0.2pt);
\fill (-0.95,0.05) circle (0.2pt);
\fill (-0.97,0.03) circle (0.2pt);

\draw (0.2,1) node {\small{$A_0$}};
\draw (0.65,0.55) node {\small{$A_1$}};
\draw (0.85,0.3) node {\small{$A_2$}};

\draw (-0.3,0.55) node {\small{$A_{-1}$}};
\draw (-0.62,0.3) node {\small{$A_{-2}$}};

\draw (1.2,0) node {\small{$J$}};

\draw (2.5,0) to (2.5,1);
\draw (3,0.5) to (2.5,0) to (2,0.5);
\draw (2.85,0) to (2.5,0) to (2.15,0);
\draw (2.65,-0.15) to (2.5,0) to (2.35,-0.15);

\fill (2.52,-0.04) circle (0.2pt);
\fill (2.54,-0.07) circle (0.2pt);
\fill (2.56,-0.1) circle (0.2pt);

\fill (2.48,-0.04) circle (0.2pt);
\fill (2.46,-0.07) circle (0.2pt);
\fill (2.44,-0.1) circle (0.2pt);

\draw (2.7,1) node {\small{$q(A_0)$}};
\draw (3.2,0.5) node {\small{$q(A_1)$}};
\draw (1.8,0.5) node {\small{$q(A_{-1})$}};
\draw (3.15,0) node {\small{$q(A_{2})$}};
\draw (1.85,0) node {\small{$q(A_{-2})$}};
\draw (2.5,-0.2) node {\tiny{$p$}};

\draw (1,-.5) node {{\bf Figure 1:} Dendrites $Z$ and $X$.};
\end{tikzpicture}
\end{center}

\bigskip

Let $g\colon Z\to Z$ be defined such that:
\begin{enumerate}
	\item $g$ is a homeomorphism;
	\item $g|_{A_n}\colon A_n\to A_{n+1}$ is a homeomorphism, for each $n\in\mathbb{Z}$;
	\item $g(J)=J$.
\end{enumerate}	

Notice that $\mathrm{Fix}(g)=\{(-1,0),(1,0)\}.$ Let $X=Z/J$ and define $f\colon X\to X$ such that $$q\circ g=f\circ q,$$ where $q\colon Z\to X$ is the quotient map. Let $p=q(J)$ (see Figure 1). It is not difficult to check that $X$ is a dendrite, $f$ is a homeomorphism and $\mathrm{Per}(f)=\{p\}$. Furthermore, $\omega_f(x)=\omega(x,f)=\{p\}$, for each $x\in X$. Therefore, $\omega_f$ is continuous. However,  $\Omega(p,f)=X$.

\bigskip

The following result shows that  the previous example is typical, more precisely,  if $f$ is a homeomorphism such that $\omega_f$ is continuous and $f$ is not equicontinuous, then $X\setminus \overline{\mathrm{Per}(f)}$ has infinitely many components and $\omega(x,f)\subseteq \overline{\mathrm{Per}(f)}$, for each $x\in X$.

\begin{theorem}\label{Th0}
	Let $X$ be a dendrite and $f\colon X\to X$ a homeomorphism. If $\omega_f$ is continuous and $X\setminus \overline{\mathrm{Per}(f)}\neq \emptyset$, then $f^{m}(W)\cap W=\emptyset$, for each component $W$ of $X\setminus \overline{\mathrm{Per}(f)}$ and each $m\in\N$. Moreover, $\omega(x,f)\subseteq \overline{\mathrm{Per}(f)}$, for each $x\in X$.
\end{theorem}

\begin{proof}
	Let $W$ be a component of $X\setminus \overline{\mathrm{Per}(f)}$. Suppose that there exists a positive integer $m$ such that $f^m(W)\cap W\neq\emptyset$. As $f$ is a homeomorphism, we have $f(\overline{\mathrm{Per}(f)})=\overline{\mathrm{Per}(f)}$ and $f^{i}(W)$ is a component of $X\setminus \overline{\mathrm{Per}(f)}$, for each $i\in\N$.  Hence, $f^m(W)=W$.
	
 Let $g=f^m|_{\overline{W}}\colon \overline{W}\to \overline{W}$. Notice that $g$ is a homeomorphism and $\overline{W}$ is a dendrite, by Theorem \ref{prop00}. So, $\mathrm{Fix}(g)\neq\emptyset$, by Theorem \ref{fixpoint}. Since $W\cap \mathrm{Per}(f)=\emptyset$, then $\mathrm{Fix}(g)\subseteq \overline{W}\setminus W$. Hence, by Theorem \ref{theo00},
	\begin{equation}\label{eq1}
	\mathrm{Fix}(g)\subseteq \mathrm{End}(\overline{W}).
	\end{equation}
	By \cite[Lemma 3.5]{Acosta}, $|\mathrm{Fix}(g)|\geq 2$. Let $x_0$ and $y_0$ be different points of $\mathrm{Fix}(g)\cap \mathrm{End}(\overline{W})$ and $\alpha$ the unique arc in $\overline{W}$ joining $x_0$ and $y_0$. Since $g$ is a homeomorphism and $\overline{W}$ is a dendrite, we have $g(\alpha)=\alpha$. Furthermore, $\alpha\setminus\{x_0,y_0\}\cap \mathrm{Fix}(g)=\emptyset$, by equation \eqref{eq1}. Notice that, for each $x\in\alpha\setminus\{x_0,y_0\}$, we have either $\lim_{i\to\infty}g^i(x)=x_0$ or $\lim_{i\to\infty}g^i(x)=y_0$.  Being $\alpha$ an arc and as the only fix points of $g$ belonging to $\alpha$ are $x_0$ and $y_0$, then without loss of generality, we may suppose that $\lim_{i\to\infty}g^i(x)=x_0$, for each $x\in\alpha\setminus\{x_0,y_0\}$, i.e., $\omega(x,g)=\{x_0\}$ for each $x\in \alpha\setminus\{x_0,y_0\}$ . Thus, $\omega(x,f)=\omega(x_0,f)$, for each $x\in \alpha\setminus\{x_0,y_0\}$. Since $\omega_f$ is continuous, $\omega(y_0,f)=\omega(x_0,f)$. We  will show  that  this  is not possible.
	
	If $m=1$, then $f(W)=W$ and $x_0,y_0\in \mathrm{Fix}(f)$, that is, $\omega(y_0,f)\neq\omega(x_0,f)$. Assume that $m>1$ and $f^l(W)\cap W=\emptyset$, for each $l<m$.
	 In view of $x_0\in \mathrm{Per}(f)$, we have $\omega(x_0,f)=\mathcal{O}(x_0,f)$ and thus $y_0\in \mathcal{O}(x_0,f)$. Hence there is $l_1\in\{1,....,m-1\}$ such that $y_0=f^{l_1}(x_0)$. Notice that $\alpha, f(\alpha),...,f^{m-1}(\alpha)$ are different arcs such that,
	 \begin{center}
	 if $f^{i}(\alpha)\cap f^{j}(\alpha)\neq\emptyset$ for $i\neq j$, then $|f^{i}(\alpha)\cap f^{j}(\alpha)|= 1$ and $f^{i}(\alpha)\cap f^{j}(\alpha)\subseteq \mathcal{O}(x_0,f)$.
	 \end{center}
	 Then, $\alpha\cap f^{l_1}(\alpha)=\{f^{l_1}(x_0)\}$. It is not difficult to see that there exist $l_1,l_2,...,l_s$ such that $f^{l_i}(\alpha)\cap f^{l_{i+1}}(\alpha)=\{f^{l_{i}}(x_0)\}$ and $\alpha\cap f^{l_{s}}(\alpha)=\{x_0\}$. So, $\alpha\cup f^{l_1}(\alpha)\cup...\cup f^{l_{s}}(\alpha)$ is a simple closed curve which is impossible since $X$ is a dendrite. Therefore, $f^{m}(W)\cap W=\emptyset$ for each $m\in\N$.
	
	 Finally, suppose that $\omega(x,f)\cap (X\setminus\overline{\mathrm{Per}(f)})\neq\emptyset$. Let $y\in \omega(x,f)\cap (X\setminus\overline{\mathrm{Per}(f)})$ be fixed and consider the component $W$ of $X\setminus\overline{\mathrm{Per}(f)}$ such that $y\in W$. Since $W$ is open and $y\in \omega(x,f)$, then there exist two different integers $l,s\in\N$ such that $f^{l}(x), f^s(x)\in W$. Assume that $s-l>0$.  Since $f^{s-l}(f^l(x))=f^s(x)$, then  $f^{s-l}(W)\cap W\neq\emptyset$, contradicting the first part of the theorem. Therefore, $\omega(x,f)\subseteq \overline{\mathrm{Per}(f)}$, for each $x\in X$.
\end{proof}

\bigskip

For our next example we need to recall the definition of  \textit{the adding machine} or \textit{odometer} $h\colon \{0,1\}^{\N}\to \{0,1\}^{\N}$. If $(a_i)_i$ is the constant sequence equal to 1, then $h((a_i)_{i})$ is the constant sequence equal to 0. Otherwise, let  $h((a_i)_{i\in\N})$ be equal to $(b_i)_{i\in\N}$ where $b_i$ is defined as follows: if $k=\min\{n\in\N : a_n=0\}$, then
\begin{gather*}
b_i=\begin{cases} 0 &\text{ if } i<k;\\ 1 &\text{ if }i=k;\\ a_i &\text{ if }i>k.\end{cases}
\end{gather*}
It is well known that $h$ is a homemorphism such that $\mathrm{Per}(h)=\emptyset$, and $\omega(x,h)=\{0,1\}^{\N}$ for each $x\in \{0,1\}^{\N}$, see for instance \cite[p.678]{Banks}.

Observe that if $w\in \{0,1\}^n$ and $[w]=\{x\in \{0,1\}^{\N} : x_i=w_i, \text{ for each }i\in\{1,...,n\}\}$, then $\{0,1\}^{\N}=\bigcup_{j=1}^{2^n}h^j([w])$, where $h^j([w])\cap h^i([w])=\emptyset$ whenever $i,j\in\{1,...,2^n\}$ and $i\neq j$. Thus, $h$ is equicontinuous.

A continuum $X$ is said to be \textit{hereditarily unicoherent} provided that $A\cap B$ is connected for any pair of subcontinua $A$ and $B$ of $X$. A \textit{dendroid} is an arcwise connected, hereditarily unicoherent continuum. Every dendrite is a dendroid. A \textit{fan} is a dendroid with only one ramification point.

Our next example shows that Theorem \ref{MainTheorem} is not true for fans; particularly, it is not valid for dendroids.

\begin{example}
       There exist a fan $X$ and a homeomorphism $f\colon X\to X$ such that $f$ is equicontinuous, and $|\mathrm{Per}(f)|=1$.
\end{example}

Let $X=(\{0,1\}^{\N}\times [0,1])/(\{0,1\}^{\N}\times \{1\})$ be the Cantor fan. We denote the unique non-degenerate class $\{0,1\}^{\N}\times \{1\}$ of $X$, by $v$. Let $h$ be the odometer map. We define $f\colon X\to X$ by: $$f(\chi)=\begin{cases} v &\text{ if }\chi=v;\\ (h((a_i)_{i\in\N}),t) &\text{ if }\chi=((a_i)_{i\in\N},t), t<1.\end{cases}$$ Since $h$ is equicontinuous, $f$ is equicontinuous. Also, it is clear that $\mathrm{Per}(f)=\{v\}$.

\bigskip

The following example shows an equicontinuous map $f\colon X\to X$, where $X$ is a dendrite, such that $\mathrm{Per}(f)\neq \bigcap_{n\in\N}f^n(X)$.

\begin{example}\label{EjemploCantor}
There exist a dendrite $X$ and a homeomorphism $f\colon X\to X$ such that $f$ is equicontinuous, and $\mathrm{Per}(f)\neq X$.
\end{example}

Let $Z=\{0,1\}^{\N}\times [0,1]$ and let $g\colon Z\to Z$ be defined by: $$g((a_i)_{i\in\N},t)=(h((a_i)_{i\in\N}),t),$$ for each $((a_i)_{i\in\N},t)\in Z$, where $h$ is the adding machine.

For each $(x,t), (y,s)\in Z$, we define $(x,t)\sim (y,s)$ if and only if $s=t$ and:
\begin{itemize}
\item if $t\in [\frac{1}{2},1]$, then $(x,t)\sim (y,t)$, for each $x,y\in \{0,1\}^{\N}$;
\item if $t\in [\frac{1}{n+2},\frac{1}{n+1})$ for some $n\in\N$, then $(x,t)\sim (y,t)$, whenever exists $w\in \{0,1\}^n$ such that $\{x,y\}\subseteq [w]$.
\end{itemize}
It is not difficult to see that $\sim$ generates an upper semicontinuous decomposition on $Z$. Thus, $X=Z/\sim$ is a continuum \cite[Theorem 3.10]{Nadler}. The dendrite $X$ is described in \cite[Example 10.39, Figure 10.39]{Nadler}. Let $q\colon Z\to X$ be the quotient map. Since $\{0,1\}^{\N}=\bigcup_{j=1}^{2^n}h^j([w])$, where $h^j([w])\cap h^i([w])=\emptyset$ whenever $i,j\in\{1,...,2^n\}$ and $i\neq j$, the correspondence \linebreak$f\colon X\to X$ defined such that 
\begin{gather*}
q\circ g=f\circ q,
\end{gather*}
is well defined and it is continuous. Observe that $f$ is equicontinuous and $\mathrm{Per}(f)=X\setminus (\{0,1\}^{\N}\times\{0\})$.

\vspace{0.5cm}

\textbf{Acknowledgements:} This research was partially supported by the grant C-2018-05 of VIE-UIS.

\end{document}